\documentclass[a4paper,11pt]{amsart}

\usepackage[T1]{fontenc}
\usepackage[utf8]{inputenc}

\usepackage{amssymb}
\usepackage{mathtools, color}

\usepackage[active]{srcltx} 
\usepackage[driverfallback=dvipdfm]{hyperref}

\makeatletter
\@namedef{subjclassname@2020}{%
  \textup{2020} Mathematics Subject Classification}
\makeatother

\allowdisplaybreaks[2]

\newtheorem{theorem}{Theorem}[section]
\newtheorem{proposition}[theorem]{Proposition}
\newtheorem{lemma}[theorem]{Lemma}
\newtheorem{corollary}[theorem]{Corollary}
\theoremstyle{definition}
\newtheorem{remark}[theorem]{Remark}

\newtheorem{example}[theorem]{Example}

\numberwithin{equation}{section}

\begin{document}

\title[Commutativity preserving mappings]{Commutativity preserving mappings in Banach algebras}

\author{M. Bre\v{s}ar}
\author{G. M. Escolano}
\author{A. M. Peralta}
\author{A. R. Villena}

\address{M. Brešar: Faculty of Mathematics and Physics, University of Ljubljana \&
Faculty of Natural Sciences and Mathematics, University of Maribor \& IMFM, Ljubljana, Slovenia}
\email{matej.bresar@fmf.uni-lj.si}

\address{G. M. Escolano, A. M. Peralta, A. R. Villena: Departamento de An\' alisis Matem\' atico, Facultad de Ciencias, Universidad de Granada, Granada, Spain}
\email{gemares@ugr.es}
\email{aperalta@ugr.es}
\email{avillena@ugr.es}




\begin{abstract}
Let $A$ and $B$ be unital complex Banach algebras having no quotients isomorphic to $\mathbb{C}$ or $M_2(\mathbb{C})$. 
Assume additionally that $B$ is semisimple. If
a surjective additive mapping $\Phi\colon A\to B$  satisfies $[\Phi(x^2),\Phi(x)] = 0$ for all $x\in A$,  
then there exist a surjective direct sum of an additive homomorphism and an additive
anti-homomorphism $\Psi\colon A\to B$,
an invertible element $\lambda\in\mathcal{Z}(B)$, and 
an additive mapping $\zeta\colon A\to\mathcal{Z}(B)$ such that 
$
\Phi(x)=\lambda\Psi(x)+\zeta(x)$ for all $x\in A$.
\end{abstract}
\maketitle

\section{Introduction} 

Let $A$ and $B$ be algebras. We say that a mapping $\Phi:A\to B$ {\em preserves commutativity} if  
for all $x,y\in A$, $[x,y]=0$ implies $[\Phi(x),\Phi(y)]=0$; here, $[x,y]$ stands for $xy-yx$.  
There is a vast literature on these mappings. 
For a historical account, we refer the reader to \cite[pp.\ 218-219]{BreCheMar} and to the recent papers 
\cite{EPV, FKS} for a more updated information.

One usually assumes that a commutativity preserver $\Phi$ is linear or at least additive, and also that it is bijective or at least surjective. Under these conditions, one wishes to show that $\Phi$ is of the form
\begin{equation}\label{ena}
\Phi(x)=\lambda\Psi(x)+\zeta(x) \quad\forall x\in A,
\end{equation}
where $\lambda$ is an invertible element from
the center $\mathcal Z(B)$ of $B$, $\zeta:A\to \mathcal Z(B)$ is an additive map, and $\psi:A\to B$
is a Jordan homomorphism.
Since we will deal  with additive (rather than linear) mappings,   by a {\em Jordan homomorphism} we will mean an additive (not necessarily linear) mapping
$\Psi:A\to B$ that satisfies $\Psi(x\circ y)=\Psi(x)\circ \Psi(y)$ for all $x,y\in A$; here, $x\circ y$ stands for $xy+yx$.

Note that $2\times 2$ matrices over a field commute if and only if they are linearly dependent modulo the scalar matrices. 
It is therefore easy to construct commutativity preserving linear mappings on  $2\times 2$ matrix algebras that are not of 
the standard form in \eqref{ena}. At the end of the paper, we will provide  a different kind of nonstandard example on a 
certain $C^*$-algebra $A$ that contains a maximal ideal $M$ such that $A/M\cong \mathbb C$ (Example \ref{ex42}). 
Therefore, when studying commutativity preservers between complex algebras, it is natural to assume that these algebras 
have no quotients isomorphic to $\mathbb C$ or
$M_2(\mathbb C)$.

The goal of this paper is to prove that a surjective commutativity preserving additive mapping $\Phi:A\to B$ is of the standard form in \eqref{ena} if
$A$ and $B$ are unital complex Banach algebras that 
satisfy the aforementioned condition on quotients and, additionally, $B$ is semisimple (see Theorem \ref{mt}). 
To the best of our knowledge, no other result on commutativity preservers in the literature considers such general
 classes of Banach algebras.

A well-established approach for  tackling additive commutativity preservers is applying the theory of 
functional identities \cite{BreCheMar}. 
Using it, one often only needs to assume
that $\Phi(x)$ always commutes with $\Phi(x^2)$.
  We will follow this approach and 
  assume only this mild condition. 
  A general result from \cite{BreCheMar} 
  (Theorem \ref{t0})  together with one of our auxiliary results (Lemma \ref{1603}) immediately shows that our main theorem  is true in the special case where the Banach algebra  $B$ is  primitive. We are assuming, however, that $B$ is semisimple, so it is natural to pass to quotients $B/P$ with primitive ideals $P$.
  The main novelty of this paper is the method used to ``glue'' the    information obtained from  Theorem \ref{t0} for  the quotients $B/P$   to the whole $B$. We believe that this method has  potential to be used elsewhere.

The principal result will be proved in Section \ref{s2}. In Section \ref{s3}, we will include a wide variety of examples of Banach algebras that satisfy the requirements of our main theorem; notably, we will reinterpret it in the case of von Neumann algebras, to show that it generalizes the result of \cite{BreMie}. In Section~\ref{s4}, we will derive a corollary on Lie homomorphisms. 
 
We finally remark that our proofs are almost entirely algebraic. The analytic setting makes it possible for us to obtain particularly neat results.

\section{Main theorem}\label{s2}

Our starting point is \cite[Theorem 7.3]{BreCheMar}.

\begin{theorem}\label{t0}
Let $J$ be a Jordan subring (of some ring), 
let $K$ be the Jordan ideal of $J$ generated by $\bigl[[J,J],[[J,J],J]\circ[[J,J],J]\bigr]$, and 
let $L$ be the Jordan ideal of $J$ generated by $K\circ K$. 
Let $Q$ be a unital ring such that $\mathcal{Z}(Q)$ is a field with $\text{\rm char}(\mathcal{Z}(Q))\ne 2$. 
Further, 
let $\Phi\colon J\to Q$ be an additive mapping satisfying $[\Phi(x^2),\Phi(x)] = 0$ for each $x\in J$. 
Suppose that $\Phi(J)$ is a $3$-free subset of $Q$, and 
suppose that $\Phi(L)\not\subseteq\mathcal{Z}(Q)$.
Then there exist a Jordan homomorphism $\Psi\colon J\to Q$, an element $\lambda\in\mathcal{Z}(Q)$, and an
additive mapping $\zeta\colon J\to\mathcal{Z}(Q)$ such that 
\[
\Phi(x)=\lambda\Psi(x)+\zeta(x) \quad\forall x\in J.
\]
\end{theorem}

\begin{remark}\label{rem22} (a)
It should be pointed out that necessarily $\lambda\ne 0$, 
for otherwise   $\Phi=\zeta$ and hence $\Phi(J)\subseteq\mathcal{Z}(Q)$, contrary to our assumption.
We also remark that the 
 standard representation in Theorem \ref{t0} is unique.

(b) The  necessity of the assumption that $\Phi(L)\not\subseteq\mathcal{Z}(Q)$ is evident from examples in \cite{BL}
and \cite{zpd}.

(c)
In  applications of Theorem \ref{t0}, the ring $Q$ is often taken to be the \emph{maximal left ring of quotients}
$Q_{ml}(R)$ of a semiprime ring $R$, in which case $\mathcal{Z}(Q_{ml}(R))$ is the so-called \emph{extended centroid},
$\mathcal{C}(R)$, of $R$. 
The interested reader is referred to \cite[Apendix A]{BreCheMar} for a rigorous exposition of $Q_{ml}(R)$
neglecting technicalities. For our purposes, however,  all we need to know is that  $Q_{ml}(R)$  is a ring 
enlargement of $R$, and that $\mathcal{C}(R)$ is  equal to $\mathbb{C}$ in the case where $R$ is a primitive 
complex Banach algebra (cf.\ \cite[Proposition 6]{Cab}).

(d) The concept of a $d$-free set occurring in Theorem \ref{t0}   is a fundamental notion in the theory of functional identities. We omit the definition and rather provide information concerning the specific situation in which  Theorem \ref{t0} will be applied. We will arrive at the case where the image of $\Phi$ is a primitive complex Banach algebra $R$. One of the basic results on functional identities states  that a primitive (or, more generally, a prime) ring  $R$ is a $d$-free subset of $Q_{ml}(R)$ if and only if not every element in $R$  is  algebraic over $\mathcal{C}(R)$  of degree less than $d$ \cite[Corollary 5.12]{BreCheMar}. 
As mentioned in (b), in the Banach algebra situation, $\mathcal{C}(R)$ is simply $\mathbb C$. We can then use either the structure theory of algebras with polynomial identities or, more directly, \cite[Lemma 5.4.1]{Au} to conclude that a primitive complex Banach algebra $R$ is a $d$-free subset of $Q_{ml}(R)$ if and only if $R$ is not isomorphic to the matrix algebra $M_k(\mathbb C)$ with $k < d$.
\end{remark}

We continue with  preliminary results.

Let $\mathbb{C}\langle\mathcal{X}\rangle$ the algebra of complex noncommutative polynomials in the
variables $\mathcal{X}=\{X_1,X_2,\ldots\}$, i.e., the free complex algebra generated by the set $\mathcal{X}$.
For each unital complex algebra $A$ and each $f(X_1,\ldots,X_n)\in\mathbb{C}\langle\mathcal{X}\rangle$, we define
$f(A)=\{f(x_1,\ldots,x_n)\colon x_1,\ldots,x_n\in A\}$.

\begin{lemma}\label{1600}
Let $f\in\mathbb{C}\langle\mathcal{X}\rangle \backslash\{0\}$, and
let $A$ be a unital complex Banach algebra with no 
quotient isomorphic to $M_n(\mathbb{C})$ with $n\le\tfrac{1}{2}\text{\rm deg}(f)$.
Then the ideal of $A$ generated by $f(A)$ is equal to $A$.
\end{lemma}

\begin{proof}
Let $I$ be the ideal of $A$ generated by $f(A)$, and suppose that $I\neq A$. Take a maximal ideal $M$ of $A$  
such that $I\subseteq M$.
Then the quotient $A/M$ is a simple Banach algebra, and it is clear that $f$ is a polynomial identity of $A/M$. We conclude from
Kaplansky's theorem \cite[Theorem 7.1.14]{Pal} that $A/M$ is isomorphic to the matrix algebra $M_n(\mathbb{C})$
with $n\le\tfrac{1}{2}\text{deg}(f)$, contrary to our assumption. 
\end{proof}

According to the standard notation, when $C$ and $D$ are subsets of an algebra $A$, the symbol $[C,D]$ will stand for the additive subgroup generated by the set of commutators $[c, d]$ with $c \in C$ and $d \in D$.

\begin{lemma}\label{1602}
Let $A$ be a unital complex Banach algebra with no  quotient isomorphic to $\mathbb{C}$.
Then the Jordan subring of $A$ generated by $[A,A]$ is equal to $A$.
\end{lemma}

\begin{proof}
Let $J$ be the Jordan subring of $A$ generated by $[A,A]$.
If $x,y\in J$, then
\[
2xy=[x,y]+x\circ y\in [A,A]+J\circ J\subseteq J+J\circ J\subseteq J.
\]
This clearly forces that $J$ is the subring of $A$ generated by $[A,A]$.

By applying Lemma \ref{1600} to the polynomial $f(X_1,X_2)=[X_1,X_2]$, we deduce that
the ideal of $A$ generated by $[A,A]$ is equal to $A$.
Hence, it follows from \cite[Theorem 3.4]{GarThi} that $A=J$. 
\end{proof}

\begin{lemma}\label{1601}
Let $f\in\mathbb{C}\langle\mathcal{X}\rangle$, and
let $A$ be a unital complex Banach algebra with no 
quotient isomorphic to $M_n(\mathbb{C})$ with $n\le\tfrac{1}{2}(\text{\rm deg}(f)+1)$.
Then 
\[
[A,A]\subseteq \text{\rm span}\,f(A),
\]
and the Jordan subring of $A$ generated by $f(A)$ is equal to $A$.
\end{lemma}

\begin{proof}
By applying Lemma \ref{1600} to the polynomial 
\[
g(x_1,\ldots,x_n,x_{n+1})=[f(x_1,\ldots,x_n),x_{n+1}],
\]
we obtain that the ideal of $A$ generated by $[f(A),A]$ is equal to $A$.
From \cite[Theorem 2.7]{Bre} (together with the comment after that theorem) we see that
\[
[A,A]\subseteq \text{span}\,f(A),
\]
which gives the first assertion of the lemma. 
Consequently, the Jordan subring  of $A$ generated by $[A,A]$ is contained in the Jordan subring of $A$
generated by $f(A)$.  On account of Lemma \ref{1602}, the former is equal to $A$.
\end{proof}

\begin{lemma}\label{1603}
Let $A$ be a unital complex Banach algebra with no quotient isomorphic to $\mathbb{C}$ or $M_2(\mathbb{C})$.
Then the Jordan subring of $A$ generated by $\bigl[[A,A],[[A,A],A]\circ[[A,A],A]\bigr]$ is equal to $A$.
\end{lemma}

\begin{proof}
Let $f,g,h\in\mathbb{C}\left\langle X_1,X_2,X_3,X_4 \right\rangle$ be defined by
\begin{equation*}
\begin{split}
f(X_1,X_2,X_3,X_4)&=[[X_1,X_2],[X_3,X_4]]
,\\
g(X_1,X_2,X_3,X_4)&=[X_1,X_2]\circ [X_3,X_4],\\
h(X_1,X_2,X_3)&=[[X_1,X_2],X_3].
\end{split}
\end{equation*}
According to Lemma \ref{1601}, $[A,A]$ is contained in each of the linear spaces $\text{span}\, f(A)$,
$\text{span}\, g(A)$, and $\text{span}\, h(A)$.
We thus have
\begin{equation*}
\begin{split}
[A,A]
&\subseteq
[[A,A],[A,A]]\\
&\subseteq
\bigl[[A,A],[A,A]\circ[A,A]\bigr]\\
&\subseteq
\bigl[[A,A],[[A,A],A]\circ[[A,A],A]\bigr],
\end{split}
\end{equation*}
which gives the result when combined with Lemma \ref{1602}. 
\end{proof}

Let $A$ and $B$ be two algebras. 
A mapping $\Psi:A\to B$ is said to be the {\em  direct sum of a homomorphism and  an anti-homomorphism} if
there exist a central idempotent $e\in B$ such that $e\Psi$ is a homomorphism
(i.e. $e\Psi$ is merely multiplicative, that is, $e\Psi(xy)=e\Psi(x)\Psi(y)$ for all 
  $x,y\in A$) and $(1-e)\Psi$ is an anti-homomorphism
  (i.e., $(1-e)\Psi(xy)=(1-e)\Psi(y)\Psi(x)$ for all 
  $x,y\in A$).

We now have enough information to prove the main result of the paper.

\begin{theorem}\label{mt}
Let $A$ and $B$ be unital complex Banach algebras, and 
let $\Phi\colon A\to B$ be a surjective additive mapping satisfying $[\Phi(x^2),\Phi(x)] = 0$ for each $x\in A$. 
Suppose that $A$ and $B$ have no quotients isomorphic to $\mathbb{C}$ or $M_2(\mathbb{C})$, and further 
suppose that $B$ is semisimple.
Then there exist a surjective Jordan  homomorphism $\Psi\colon A\to B$,
an invertible element $\lambda\in\mathcal{Z}(B)$, and 
an additive mapping $\zeta\colon A\to\mathcal{Z}(B)$ such that 
\[
\Phi(x)=\lambda\Psi(x)+\zeta(x) \quad\forall x\in A.
\]
Furthermore, $\Psi$ is the direct sum of a homomorphism and  an anti-homo\-morphism.
\end{theorem}

\begin{proof}
The proof consists in the construction of an invertible element $\alpha$ in $\mathcal{Z}(B)$ and 
an additive mapping $\beta\colon A\to\mathcal{Z}(B)$
such that the mapping
\[
x\mapsto\alpha\Phi(x)+\beta(x)
\]
is a Jordan  homomorphism from $A$ to $B$.

Let $\text{Prim}(B)$ denote the set of all primitive ideals of $B$.
For each $P\in\text{Prim}(B)$, we write $\pi_P$ for the quotient mapping from $B$ onto $B/P$, and, for all $a,b\in B$,
we write $a\equiv_P b$ if $a-b\in P$.

Take  $P\in\text{Prim}(B)$. Clearly,
$\pi_P\Phi$ is a surjective additive mapping from $A$ onto $B/P$ that satisfies
$[(\pi_P\Phi)(x^2),(\pi_P\Phi)(x)] = 0$ for each $x\in A$. 
 Theorem \ref{t0},
 along with Lemma \ref{1603} (which shows that $L=A$ in the mentioned theorem) and the information provided in Remark \ref{rem22}, 
assures the existence of a Jordan homomorphism $\Psi_P\colon A\to Q_{ml}(B/P)$, 
an element $\lambda_P\in\mathbb{C}\setminus\{0\}$, and 
an additive functional $\zeta_P\colon A\to\mathbb{C}$ 
such that 
\[
(\pi_P\Phi)(x)=\lambda_P\Psi_P(x)+\zeta_P(x)\pi_P(1) \quad\forall x\in A.
\]
Set $\alpha_P=\lambda_P^{-1}$ and $\beta_P=-\lambda_P^{-1}\zeta_P$. Then
\[
\Psi_P(x)=\alpha_P(\pi_P\Phi)(x)+\beta_P(x)\pi_P(1)\quad\forall x\in A.
\]
(Thus, $\Psi_P$ actually maps to $B/P$).

Our next objective is to show that there exists an element $\alpha\in B$ such that
\begin{equation}\label{1744}
\pi_P(\alpha)=\alpha_P\pi_P(1)\quad\forall P\in\text{Prim}(B).
\end{equation}
From the semisimplicity of $B$ it may be concluded that such an element is necessarily unique and 
belongs to the centre of $B$.

For each $x\in A$ and each $P\in\text{Prim}(B)$, we have
\begin{equation*}
\begin{split}
\alpha_P\Phi(x^2)+\beta_P(x^2)1
&\equiv_P
\bigl(\alpha_P\Phi(x)+\beta_P(x)1\bigr)^2\\
&=
\alpha_P^2\Phi(x)^2+
2\alpha_P\beta_P(x)\Phi(x)+
\beta_P(x)^2 1.
\end{split}
\end{equation*}
Commuting with an arbitrary 
 $c\in B$,  we arrive at
\[
\alpha_P[\Phi(x^2),c]\equiv_P
\alpha_P^2[\Phi(x)^2,c]+
2\alpha_P\beta_P(x)[\Phi(x),c],
\]
whence
\[
[\Phi(x^2),c]\equiv_P
\alpha_P[\Phi(x)^2,c]+
2\beta_P(y)[\Phi(x),c].
\]
Now, commuting with   $[\Phi(x),c]$, we obtain
\begin{equation}\label{1743}
\bigl[[\Phi(x^2),c],[\Phi(x),c]\bigr]\equiv_P
\alpha_P\bigl[[\Phi(x)^2,c],[\Phi(x),c]\bigr].
\end{equation}
By applying Lemma \ref{1600} to 
the polynomial $$f(X_1,X_2)=\bigl[[X_1^2,X_2],[X_1,X_2]\bigr]\in\mathbb{C}\langle X_1,X_2\rangle,$$
we obtain that the ideal of $B$ generated by $f(B)$ is equal to $B$. Hence, there exist an $n\in\mathbb{N}$
and $a_k,b_k,c_k,d_k\in B$ $(k=1,\ldots,n)$ such that
\begin{equation}\label{1740}
1=\sum_{k=1}^n a_k\bigl[[b_k^2,c_k],[b_k,c_k]\bigr]d_k.
\end{equation}
For each $k\in\{1,\ldots,n\}$, by the surjectivity of $\Phi$, we can take $x_k\in A$ such that $\Phi(x_k)=b_k$ and 
define
\[
\alpha=\sum_{k=1}^n a_k\bigl[[\Phi(x_k^2),c_k],[\Phi(x_k),c_k]\bigr]d_k\in B.
\] 
From \eqref{1743} we deduce that
\begin{equation*}
\begin{split}
\pi_P(\alpha)
&=
\sum_{k=1}^n\pi_P(a_k)\pi_P\bigl(\bigl[[\Phi(x_k^2),c_k],[\Phi(x_k),c_k]\bigr]\bigr)\pi_P(d_k)\\
&=
\sum_{k=1}^n\pi_P(a_k)\alpha_P\pi_P\bigl(\bigl[[b_k^2,c_k],[b_k,c_k]\bigr]\bigr)\pi_P(d_k)\\
&=
\alpha_P\pi_P\Bigl(\sum_{k=1}^n a_k\bigl[[b_k^2,c_k],[b_k,c_k]\bigr]d_k\Bigr) =
\alpha_P\pi_P(1),
\end{split}
\end{equation*}
for each $P\in\text{Prim}(B)$. This proves \eqref{1744}.

We next show that $\alpha$ is invertible. Otherwise, $\alpha B$ is an ideal of $B$ different from $B$, 
and therefore we can take a maximal ideal $M$ of $B$ such that $\alpha B\subseteq M$.
Since $\alpha\in M$, we have $\pi_M(\alpha)=0$, but, on the other hand, $\pi_M(\alpha)=\alpha_M\neq 0$, a contradiction.

Having disposed of the element $\alpha$, we can now address the problem of finding the mapping $\beta$ through the family 
$\{\beta_P\}_{P\in\text{Prim}(B)}$.
We define the set $A_0$ to consists of those $x\in A$ for which
there exists an element $\beta(x)\in B$ such that
\[
\beta_P(x)\pi_P(1)=\pi_P(\beta(x))\quad\forall P\in\text{Prim}(B).
\]
It should be pointed out that such an element $\beta(x)$ is necessarily unique and belongs to $\mathcal{Z}(B)$.

We claim that $A_0$ is a Jordan subring of $A$. It is straightforward to check that
$A_0$ is an additive subgroup; moreover,
since for all $x,y\in A_0$,
\begin{align*}
\beta_P(x+y)\pi_P(1) =\beta_P(x)\pi_P(1)+\beta_P(y)\pi_P(1)
=
\pi_P(\beta(x) + \beta(y)), 
\end{align*}
we see that
$\beta(x+y) =\beta(x) +\beta(y)$, i.e.,
$\beta$ is an additive mapping.
Further, for all $x,y\in A_0$,
\begin{align*}
   &\alpha\Phi(x\circ y)+\beta_P(x\circ y)1\\ \equiv&_P \bigl(\alpha\Phi(x)+\beta_P(x)1\bigr)\circ\bigl(\alpha\Phi(y)+\beta_P(y)1\bigr)\\\equiv&_P \bigl(\alpha\Phi(x)+\beta(x)1\bigr)\circ\bigl(\alpha\Phi(y)+\beta(y)1\bigr)\\
=& \alpha^2\Phi(x)\circ\Phi(y)+\alpha\beta(x)\Phi(y)+\alpha\beta(y)\Phi(x)+\beta(x)\beta(y)1,
\end{align*}
which shows that $\beta_P(x\circ y)\pi_P(1)$ coincides with
\begin{align*}  \pi_P\Bigl(\alpha^2\Phi(x)\circ\Phi(y) +
\alpha\beta(x)\Phi(y)+\alpha\beta(y)\Phi(x) + \beta(x)\beta(y)1-\alpha\Phi(x\circ y)
\Bigr),
\end{align*}
and therefore $x\circ y\in A_0$. This proves our claim.

Now, since 
\[
\bigl[[x,y],z\bigr]=(y\circ z)\circ x-y\circ(z\circ x)\quad\forall x,y,z\in A,
\]
and $\pi_P\circ(\alpha\Phi)+\beta_P\pi_P(1)$ is a Jordan homomorphism, we have
\begin{align*}
&\alpha\Phi\bigl([[x,y],z]\bigr)+\beta_P\bigl([[x,y],z]\bigr)1\\\equiv&_P
\bigl[
[\alpha\Phi(x)+\beta_P(x)1,\alpha\Phi(y)+\beta_P(y)1],\alpha\Phi(z)+\beta_P(z)1
\bigr]\\=&
\alpha^3\bigl[[\Phi(x),\Phi(y)],\Phi(z)\bigr],    
\end{align*}
and thus
\[
\beta_P\bigl([[x,y],z]\bigr)\pi_P(1)=
\pi_P\bigl(\alpha^3\bigl[[\Phi(x),\Phi(y)],\Phi(z)\bigr]-\alpha\Phi\bigl([[x,y],z]\bigr)\bigr),
\]
 showing that
\begin{equation}\label{1611}
\bigl[[x,y],z\bigr]\in A_0\quad\forall x,y,z\in A.
\end{equation}
Hence the Jordan subring $J$ of $A$ generated by $\bigl[[A,A],A\bigr]$ is contained in $A_0$.
However, Lemma \ref{1603} implies that $J=A$, which yields $A_0=A$.

Now define
\[
\Psi :=\alpha\Phi+\beta.
\]
Clearly, $\Psi$ is an additive mapping and since 
$\pi_P\Psi=\Psi_P$ is a Jordan  homomorphism, we have
\[
\pi_P\bigl(
\Psi(x\circ y)-\Psi(x)\circ\Psi(y)
\bigr)\!=\!
\Psi_P(x\circ y)-\Psi_P(x)\circ\Psi_P(y)=
0 \ \forall P\in\text{Prim}(B).
\]
Since $B$ is semisimple, we conclude that $\Psi$ is a Jordan homomorphism.

Note that by setting
\[
\lambda=\alpha^{-1}, \ \
\zeta=-\alpha^{-1}\beta,
\]
we arrive at the representation
\[
\Phi=\lambda\Psi+\zeta.
\]

Let us show that $\Psi$ is surjective.
Since $\Psi$ is a Jordan  homomorphism, we see that $\Psi(A)$ is a Jordan subring of $B$.
Further, since $\Phi$ is surjective, $\beta$ is centre-valued, and $\alpha$ is invertible, it follows that
\begin{align*}
\Psi\bigl([[A,A],A]\bigr) &=
\bigl[[\Psi(A),\Psi(A)],\Psi(A)\bigr]=
\bigl[[\alpha\Phi(A),\alpha\Phi(A)],\alpha\Phi(A)\bigr]\\
&=
\bigl[[\alpha B,\alpha B],\alpha B\bigr]=
\bigl[[B,B],B\bigr],
\end{align*}
and hence
\[
\bigl[[B,B],B\bigr]\subseteq\Psi(A).
\]
Consequently, the Jordan subring $H$ of $B$ generated by $[[B,B],B]$ is contained in $\Psi(A)$.
By Lemma \ref{1603} and the conclusion in its proof for the polynomial $h$, $H=B$ and hence $\Psi(A)=B$.

Using either 
 \cite[Theorem 2.1]{Sin} or
 \cite[Theorem 2.10]{BZ}, we see that $\Psi$ is the direct sum of a homomorphism and  an anti-homomorphism.
\end{proof}

\section{Examples}\label{s3}

This section presents a wide range of natural examples to which our main conclusion applies. The culminating point is the case of surjective additive mappings preserving commutativity only in one direction between von Neumann algebras with no central summands of type $I_1$ and $I_2$.

\begin{example}[Simple unital Banach algebras]

An obvious example of semisimple unital complex Banach algebra which has no quotients isomorphic to $\mathbb{C}$ or $M_2(\mathbb{C})$ is a simple unital Banach algebra $A$ with $\dim(A)>4$. It should be noted that, for a unital complex Banach algebra $A$, 
it does not matter if we think of simplicity in an algebraic sense, i.e. $0$ and $A$ are the only ideals of $A$, or in a
topological sense, i.e. $0$ and $A$ are the only closed ideals of $A$. Classical examples are the matrix algebra 
$M_n(\mathbb{C})$ for each $n\in\mathbb{N}$ $n\geq 3$, and 
the Calkin algebra $\mathcal{C}(H)$ on a separable infinite-dimensional complex Hilbert space $H$, but 
there are many others such as 
the Cuntz algebra $\mathcal{O}_n$ with $2\le n\le\infty$, 
$UHF$ algebras, noncommutative tori (cf. \cite{Bla}). In fact, the study of simple $C^*$-algebras is one of the main thrusts of the classification program of $C^*$-algebras.
\end{example}

\begin{example}[Properly infinite unital Banach algebras]\label{pi}

Let $A$ be a complex algebra.
Two idempotents $p,q\in A$ are \emph{algebraically equivalent},
written $p\approx q$, if there exist $a,b\in A$ such that $p=ab$ and $q=ba$. A unital complex algebra $A$ is \emph{properly infinite}
if there exist idempotents $p,q\in A$ such that $pq=0=qp$ and satisfy $p\approx 1\approx q$
(cf. \cite[Definitions 1.3.18 and 1.3.21]{D}).
The Banach algebra $\mathcal{B}(\mathcal{X})$ of all bounded linear operators on a complex Banach space $\mathcal{X}$ 
is properly infinite if and only if $\mathcal{X}$ contains a complemented subspace isomorphic to $\mathcal{X}\oplus\mathcal{X}$ 
(see \cite[Proposition 1.9]{L}).
The classical Banach spaces $C([0, 1])$, $c_0$, $\ell^p$ and $L^p([0, 1])$, where $1\le p\le\infty$, 
are isomorphic to the direct sum of two copies of themselves. Hence, they are examples of Banach spaces 
$\mathcal{X}$ for which $\mathcal{B}(\mathcal{X})$ is properly infinite.
For a $C^*$-algebra $A$,  rather than considering idempotents, one considers \emph{projections}.
Further, the natural equivalence of projections is that of \emph{Murray-von Neumann equivalence}, which says that
two projections $p,q\in A$ are equivalent, written $p\sim q$, if there exists $u\in A$ such that $p=u^*u$ and $q=uu^*$ 
(hence, $u$ is a partial isometry). Each idempotent is algebraically equivalent to a projection and, moreover, given
projections $p,q\in A$ we have $p\sim q$ if and only if $p\approx q$ (cf. \cite[Propositions 1.10.21(i) and 3.2.10]{D}).
By \cite[Proposition 2.4]{DH},
a unital $C^*$-algebra $A$ is properly infinite in the algebraic sense if and only if it is properly infinite
in the $C^*$ sense, i.e. there exist mutually orthogonal projections $p,q\in A$ such that $p\sim 1\sim q$.

Let $A$ be a properly infinite unital complex Banach algebra. We claim that $A$ has no quotients isomorphic to $M_d(\mathbb{C})$ for any positive integer $d$. To derive a contradiction, suppose that there exists a surjective homomorphism $\pi\colon A\to M_d(\mathbb{C})$
for some positive integer $d$.
Let $p,q\in A$ idempotents such that $pq=0=qp$ and $p\approx 1\approx q$.
Then 
$\pi(p)$ and $\pi(q)$ are idempotents of $M_d(\mathbb{C})$, 
$\pi(p)\pi(q)=0=\pi(q)\pi(p)$, and, 
$\pi(p)\approx 1\approx\pi(q)$ (because  $\pi(1)=1$),
which entails that $M_d(\mathbb{C})$ is a properly infinite $C^*$-algebra, a contradiction.
\end{example}

\begin{example}[Stability under unital homomorphisms]

Let $A$ and $B$ be unital complex Banach algebras, and let $\Phi\colon A\to B$ a unital homomorphism.

Suppose that $B$ has a quotient isomorphic to $M_d(\mathbb{C})$ for some $d\in\mathbb{N}$.
We claim that $A$ has  a quotient isomorphic to $M_n(\mathbb{C})$ for some $n\in\{1,\ldots,d\}$.
Let $\pi\colon B\to M_d(\mathbb{C})$ be a surjective homomorphism. Then $\ker(\pi\Phi)$ is
an ideal of $A$ different from $A$ (because $\pi(\Phi(1))=1$) and therefore there exists a maximal ideal $M$ of $A$ such that
$\ker(\pi\Phi)\subseteq M$. Since
$\frac{A}{M}\cong \frac{A/\ker(\pi\Phi)}{M/\ker(\pi\Phi)}$ and $\frac{A}{\ker(\pi\Phi)}\cong \pi(\Phi(A))\subseteq M_d(\mathbb{C})$,
we see that $\dim A/M\le d^2$  and hence $A/M\cong M_n(\mathbb{C})$ for some $n\in\{1,\ldots d\}$.

Consequently, if $A$ has no quotients isomorphic to $\mathbb{C}$ or $M_2(\mathbb{C})$, then
$B$ has no quotients isomorphic to $\mathbb{C}$ or $M_2(\mathbb{C})$.
This simple and deceptively innocuous observation turns out to be extremely powerful to construct unital Banach algebras
satisfying the conditions of Theorem \ref{mt}. To illustrate this, we will provide a wide range of examples below.
\begin{enumerate}
\item[\rm(i)]
\emph{Quotients.}
For each unital complex Banach algebra $A$ and each closed ideal $I$ of $A$ different from $A$, the quotient mapping
$\pi_I\colon A\to A/I$ is a unital homomorphism. Thus,  
if $A$ has no quotients isomorphic to $\mathbb{C}$ or $M_2(\mathbb{C})$, then
$A/I$ has no quotients isomorphic to to $\mathbb{C}$ or $M_2(\mathbb{C})$.
\item[\rm(ii)] 
\emph{Tensor products.}
Let $A$ and $B$ be unital complex Banach algebras.
Then both $A$ and $B$ embed unitally into the projective tensor product $A\widehat{\otimes}B$ through
the mappings $a\mapsto a\otimes 1_B$ and $b\mapsto 1_A\otimes b$, respectively. Therefore, if any of the algebras
has no quotients isomorphic to  $\mathbb{C}$ or $M_2(\mathbb{C})$, then
$A\widehat{\otimes}B$ has no quotients isomorphic to $\mathbb{C}$ or $M_2(\mathbb{C})$.
However, it must be noted that the semisimplicity does not always carry over from $A$ and $B$ to $A\widehat{\otimes}B$.
\item[\rm(iii)]
\emph{Matrix algebras.}
An immediate consequence of the preceding example is that the matrix algebra $M_n(A)$ 
has no quotients isomorphic to $\mathbb{C}$ or $M_2(\mathbb{C})$ for each unital complex Banach algebra $A$ and each $n\ge 3$.
Further, $M_n(A)$ is semisimple whenever $A$ is semisimple.
\item[\rm(iv)]
\emph{Banach algebras of vector-valued functions.}
Let $A$ be a unital complex Banach algebra and let $S$ be a non-empty set.
Then $A^S$ is a unital complex algebra with respect to the pointwise operations. 
Let $B$ be a unital subalgebra of $A^S$, and suppose that $B$ contains all the constant functions. 
Assume additionally that $B$ is equipped with a norm that makes it a Banach algebra.
Then $A$ embeds unitally into $B$ through the mapping $a\mapsto (s\mapsto a)$. 
Consequently, if $A$ has no quotients isomorphic to $\mathbb{C}$ or $M_2(\mathbb{C})$, then the same is true of $B$.
Further, for each $s\in S$, the evaluation mapping $\varepsilon_s\colon B\to A$ defined by $\varepsilon_s(f)=f(s)$
$(f\in B)$ is a surjective homomorphism and hence $\varepsilon_s(\text{Rad}(B))\subseteq\text{Rad}(A)$. 
Therefore, $B$ is semisimple whenever $A$ is semisimple. This makes it possible for us to show, for example, that any of the following unital Banach algebras of $A$-valued vector functions
have no quotients isomorphic to $\mathbb{C}$ or $M_2(\mathbb{C})$ provided that $A$ does not have them either:
\begin{itemize}
\item
The Banach algebra $\ell^\infty(S, A)$ of all $A$-valued bounded functions on the set $S$;
\item
The Banach algebra $C(K,A)$ of $A$-valued continuous functions on a compact Hausdorff space $K$;
\item
The Banach algebra $C^n(I,A)$ of all $n$-times continuously differentiable $A$-valued functions on a compact interval $I$, 
for each $n\in\mathbb{N}$;
\item
The $A$-valued Lipschitz algebras $\text{Lip}_\alpha(K,A)$ and $\text{lip}_\alpha(K,A)$ 
for each non-empty compact metric space $K$ and each $0<\alpha\le 1$.
\end{itemize}
Further, each of these algebras is semisimple given that $A$ is so.
\end{enumerate}
\end{example}


\begin{example}[von Neumman algebras with no central summands of type $I_1$ or $I_2$]

The first author and Miers established in \cite{BreMie} that every additive bijection between von Neumann algebras with no central summands of type $I_1$ or $I_2$ preserving commutativity in both directions is of the standard form in \eqref{ena}.
In order to compare Theorem \ref{mt} with the  result of \cite{BreMie}, it is convenient to record the following proposition.

\begin{proposition}\label{pr}
Let $\mathcal{M}$ be a von Neumann algebra, and let $d$ be a positive integer. 
Then the following assertions are equivalent:
\begin{enumerate}
\item[\rm(i)] 
$\mathcal{M}$ has a quotient isomorphic to $M_d(\mathbb{C})$;
\item[\rm(ii)] 
$\mathcal{M}$ has a central summand of type $I_d$.
\end{enumerate}
\end{proposition}

\begin{proof}
(i)$\implies$(ii) 
Suppose that $\mathcal{I}$ is an ideal of $\mathcal{M}$ such that the quotient $\mathcal{M}/\mathcal{I}$ 
is isomorphic to $M_d(\mathbb{C})$, and we think of the quotient mapping $\pi_\mathcal{I}$ as a mapping 
from $\mathcal{M}$ onto $M_d(\mathbb{C})$. It should be pointed out that $\mathcal{I}$ is a maximal ideal.

Let us first examine several particular cases.

\emph{Case 0.}
Suppose that $\mathcal{M}$ is a type $I_n$ von Neumann algebra with $n<d$.
Take $p_1,\ldots,p_n$ mutually orthogonal and mutually equivalent abelian projections in $\mathcal{M}$ such that
$p_1+\cdots+p_n=1$. Then $\pi_\mathcal{I}(p_1),\ldots,\pi_\mathcal{I}(p_n)$ are mutually orthogonal
and mutually equivalent abelian projections in $M_d(\mathbb{C})$ such that
$\pi_\mathcal{I}(p_1)+\cdots+\pi_\mathcal{I}(p_n)=\pi_\mathcal{I}(1)=1$, which is impossible.

The next cases are based on the following observation.
If $p_1,\ldots,p_{d+1}$ are mutually orthogonal and mutually equivalent projections, then
$p_1+\cdots+p_{d+1}\in\mathcal{I}$. Indeed, $\pi_\mathcal{I}(p_1),\ldots,\pi_\mathcal{I}(p_{d+1})$ are mutually
orthogonal projections in $M_d(\mathbb{C})$ and therefore some of them is necessarily zero. Further, the projections
$\pi_\mathcal{I}(p_1),\ldots,\pi_\mathcal{I}(p_{d+1})$ are also mutually equivalent and, in consequence, all of them
are zero. This implies that $p_1,\ldots,p_{d+1}\in\mathcal{I}$ and hence $p_1+\cdots+p_{d+1}\in\mathcal{I}$.

\emph{Case 1.}
Suppose that $\mathcal{M}$ has no central portion of type $I$. 
From \cite[Lemma 6.5.6]{KR} it follows that there exist $p_1,\ldots,p_{d+1}$ 
mutually orthogonal and mutually equivalent projections such that $p_1+\dots+p_{d+1}=1$.
From what has previously been observed, it may be concluded that $1\in\mathcal{I}$ and therefore that $\mathcal{I}=\mathcal{M}$,
which is impossible.

\emph{Case 2.}
Suppose that $\mathcal{M}=\ell^\infty-\bigoplus_{n\in N}\mathcal{M}_n$, 
where $N\subseteq\{n\in\mathbb{N}\colon n\ge d\}$ and
$\mathcal{M}_n$ is a von Neumann algebra of type $I_n$ for each $n\in N$.
For each $n\in N$ take $j_n\in\mathbb{N}$ and $k_n\in\{0,1,\ldots,d\}$ such that
\[
n=(d+1)j_n+k_n,
\]
and take $p_{n,1},\ldots,p_{n,n}$ mutually orthogonal and mutually equivalent abelian projections in $\mathcal{M}_n$ such that
\[
p_{n,1}+\cdots+p_{n,n}=1_{\mathcal{M}_n}.
\]
We define projections $p_1,\ldots,p_{d+1},q$ by
\[
p_m=\sum_{n\in N}\sum_{j=0}^{j_n-1}p_{n,m+(d+1)j}\quad m\in\{1,\ldots,d+1\},
\]
\[
q=\sum_{n\in N\colon k_n\ne 0}\sum_{k=1}^{k_n}e_{n,(d+1)j_n+k}.
\]
Observe that $q$ is possibly zero, 
\[
p_1+\cdots+p_{d+1}+q=1,
\]
$p_1,\ldots,p_{d+1}$ are mutually equivalent and hence $p_1+\cdots+p_{d+1}\in\mathcal{I}$, 
and further, $q$ is Murray-von Neumann equivalent to a subprojection of $p_1,\ldots, p_{d+1}$, 
so a similar argument to that given above also shows that $q\in \mathcal{I}$. 
We thus get $1\in\mathcal{I}$, a contradiction.

Having disposed of these preliminary cases, we can now address the general case.
According to the type decomposition of the von Neumann algebras, there exist mutually orthogonal central projections
$z_1,\ldots,z_d,w_1,w_2,w_3$ such that
\[
z_1+\cdots+z_d+w_1+w_2+w_3=1,
\]
$z_k\mathcal{M}$ is of type $I_k$  or $z_k=0$ ($k=1,\ldots,d$),
$w_j\mathcal{M}$ is a von Neumann algebra as considered in Case $j$ above, or $w_j=0$ ($j=1,2$),
$w_3\mathcal{M}$ is a properly infinite von Neumann algebra or $w_3=0$.
Note that if $z$ is a central projection of $\mathcal{M}$, then $\pi_\mathcal{I}(z)$ is a central projection of $M_n(\mathbb{C})$
and therefore $\pi_\mathcal{I}(z)\in\{0,1\}$. Since
\[
\pi_\mathcal{I}(z_1)+\cdots+\pi_\mathcal{I}(z_d)+
\pi_\mathcal{I}(w_1)+\pi_\mathcal{I}(w_2)+\pi_\mathcal{I}(w_3)=1,
\]
we deduce that all the summands but exactly one are zero. 
Suppose that $\pi_\mathcal{I}(w_j)\neq 0$ for some $j\in\{1,2,3\}$. Then
\[
w_j\mathcal{M}/(w_j\mathcal{M}\cap\mathcal{I})\cong
(w_j\mathcal{M}+\mathcal{I})/\mathcal{I}\cong
\mathcal{M}/\mathcal{I}\cong M_d(\mathbb{C}),
\]
and, on account of the cases $1,2$ and Example \ref{pi}, this is impossible.
On the other hand, suppose that $\pi_\mathcal{I}(z_k)\neq 0$ for some $k\in\{1,\ldots,d-1\}$.
As before, this yields  $z_k\mathcal{M}/(z_k\mathcal{M}\cap\mathcal{I})\cong M_d(\mathbb{C})$,
which, according to case $0$, is impossible.
Consequently, the only possible case is that $\pi_\mathcal{I}(z_d)\neq 0$, and hence $z_d\neq 0$, which proves that 
$\mathcal{M}$ has a central summand of type $I_d$.

(ii)$\implies$(i)
Let $e$ be a central projection such that $e\mathcal{M}$ is a type $I_n$ von Neumann algebra.
Then $e\mathcal{M}$ can be thought of as $C(K,M_n(\mathbb{C}))$ for some hyper--Stonean space. 
Take $t\in K$ and define 
\[
I=\bigl\{x\in\mathcal{M}\colon (ex)(t)=0\bigr\}.
\]
Then $I$ is a closed (in norm) ideal of $\mathcal{M}$ and 
the mapping $x\mapsto (ex)(t)$ from $\mathcal{M}$ onto $M_n(\mathbb{C})$ 
drops to an isomorphism from $\mathcal{M}/I$ onto $M_n(\mathbb{C})$.
\end{proof}

Combining Theorem \ref{mt} and Proposition \ref{pr}, we obtain the following corollary.

\begin{corollary}\label{mtcvn}
Let $\mathcal M$ and $\mathcal N$ be von Neumann algebras with no central summands of type $I_1$ and $I_2$.
If  a surjective additive mapping $\Phi: \mathcal M\to\mathcal N$ satisfies $[\Phi(x^2),\Phi(x)] = 0$ for each $x\in \mathcal M$, 
then there exist a surjective Jordan   homomorphism $\Psi\colon \mathcal M\to \mathcal N$,
an invertible element $\lambda\in\mathcal{Z}(\mathcal N)$, and 
an additive mapping $\zeta\colon A\to\mathcal{Z}(\mathcal N)$ such that 
\[
\Phi(x)=\lambda\Psi(x)+\zeta(x) \quad\forall x\in \mathcal M.
\]
Furthermore, $\Psi$ is the direct sum of a homomorphism and  an anti-homo\-morphism.
\end{corollary}

This corollary generalizes the main result of \cite{BreMie} in which $\Phi$ is assumed to be bijective and $\Phi^{-1}$ is  also  assumed to preserve commutativity (of some special pairs of elements). From \cite[Example 7.17]{zpd} it is evident that the latter is not automatically fulfilled; more precisely, $\Phi$ in Corollary \ref{mtcvn} may be bijective but $\Phi^{-1}$ does not preserve commutativity (in fact, does not map the center of $\mathcal N$ to the center of $\mathcal M$).
\end{example}

\section{Lie homomorphisms}\label{s4}

Recall that an additive mapping $\Phi$ from a ring $A$ to a ring $B$ is called a {\em Lie homomorphism} if 
\begin{equation}\label{lh}
    \Phi([x,y])=[\Phi(x),\Phi(y)]\quad\forall x,y\in A.
\end{equation}
Clearly, $\Phi$ preserves commutativity.

Obvious examples of Lie homomorphisms are homomorphisms and the negatives of  anti-homo\-morphisms. 
A more general example is the  {\em  direct sum of a homomorphism and the negative of an anti-homomorphism}, i.e.,
a mapping $\Theta:A\to B$ for which  
there exists a central idempotent $e\in B$ such that $e\Theta$ is a  homomorphism
 and $(1-e)\Theta$ is the negative of an anti-homomorphism
  (i.e., $(1-e)\Theta(xy)=-(1-e)\Theta(y)\Theta(x)$ for all 
  $x,y\in A$). 

Another standard example is any 
additive mapping $\zeta$ from
$A$ 
 to  $\mathcal Z(B)$ that vanishes on commutators (i.e., $\zeta([x,y])=0$
for all $x,y\in A$). Observe that the sum of an arbitrary Lie homomorphism and such a mapping $\zeta$ is again a Lie homomorphism.

It is a classical problem to find conditions under which   the above examples are also the only examples of Lie homomorphisms. Obviously,   we can  tackle it by
using Theorem \ref{mt}. 
 
\begin{corollary}\label{mc}
Let $A$ and $B$ be unital complex Banach algebras, and 
let $\Phi\colon A\to B$ be a surjective Lie homomorphism. 
Suppose that $A$ and $B$ have no quotients isomorphic to $\mathbb{C}$ or $M_2(\mathbb{C})$, and further 
suppose that $B$ is semisimple.
Then $\Phi=\Theta+\zeta$ where
$\Theta:A\to B$ is a surjective direct  sum of a homomorphism and the negative of an anti-homomorphism and $\zeta:A\to \mathcal Z(B)$ is an additive mapping  vanishing on commutators.
\end{corollary}

\begin{proof} 
Let $\Psi$, $\zeta,$ $\lambda$, and $e$ denote the maps, the invertible central element and the central projection obtained by applying Theorem \ref{mt} to $\Phi$. By applying the hypothesis on $\Phi$ (see \eqref{lh}), we obtain
   \begin{equation}\label{lh2}
       \lambda \Psi([x,y]) + \zeta([x,y]) = \lambda^2 [\Psi(x),\Psi(y)]\quad\forall x,y\in A.
   \end{equation}
    
    Multiplying \eqref{lh2} by $e$ and using that $e\Psi$ is a homomorphism it follows
    that 
    \begin{equation}\label{lh3}
       (\lambda^2 - \lambda)e[\Psi(x),\Psi(y)] = e\zeta([x,y])\in \mathcal Z(B) 
\quad\forall x,y\in A.    \end{equation}
    Since $\Psi$ is surjective and $\lambda$ is invertible, this yields
    $$(\lambda - 1)e[[B,B],B]=\{0\}.$$
Lemma \ref{1601} shows that the algebra $B$ is generated by  $[[B,B],B]$. Consequently,
\begin{equation}\label{ee}
    (\lambda - 1)e=0,
\end{equation}
and hence
 \begin{equation} \label{ff}
e\zeta([x,y])=0 
\quad\forall x,y\in A, 
\end{equation}
by \eqref{lh3}.

Similarly, by multiplying \eqref{lh2} by $1-e$ and using that $(1-e)\Psi$ is an anti-homomorphism, we obtain
\begin{equation}\label{ee2}
    (\lambda + 1)(1-e)=0,
\end{equation}
and \begin{equation} \label{ff2}
(1-e)\zeta([x,y])=0 
\quad\forall x,y\in A. \end{equation}
From \eqref{ee} and \eqref{ee2} we infer that
$$\Theta:=\lambda\Psi = e\Psi - (1-e)\Psi$$
is the direct sum of a homomorphism and the negative of an anti-homomor\-phism, and from \eqref{ff} and \eqref{ff2} we infer that $\zeta$ vanishes on commutators.
\end{proof}

The following example justifies the assumption that our algebras have no quotients isomorphic to $\mathbb C$. It is similar to \cite[Example 5.1]{Brev} (see also \cite[Example 7.4]{BZ2}).

\begin{example} \label{ex42}Let $H$ be an  infinite-dimensional, separable Hilbert space, let $K(H)$ be the $C^*$-algebra of  all compact operators on $H$, and let 
 $A$ be the $C^*$-algebra  obtained by adjoining a unity to the direct product $M=K(H)\times K(H)$. Note  that $M$
is a maximal ideal of $A$ and $A/M\cong \mathbb C$.

 Define $\Phi:A\to A$ by
$$\Phi(\lambda \mathbf{1} + (a,b)) = \lambda \mathbf{1} + (a, -b^*)\quad \forall \lambda\in \mathbb C,\,\,\,\forall a,b\in K(H).$$
 One easily checks that 
$\Phi$ is a Lie automorphism of $A$. 

We will show not only that $\Phi$ does not satisfy the conclusion of Corollary \ref{mc}, but that it is not standard in the sense of \cite{BKZ}.
Suppose  this is not true, i.e., suppose $\Phi$ is standard.
 Then there exist a homomorphism
 $\psi_1:A\to A$ and an anti-homomorphism 
 $\psi_2:A\to A$
such that 
\[
\psi_1(A)\cdot \psi_2(A)=\psi_2(A)\cdot \psi_1(A) = \{0\},
\] 
and  $\Phi$ coincides with $\psi_1-\psi_2$ on $[A,A]$.

Let $F(H)$ denote the algebra of all finite rank operators in $K(H)$.
Suppose $\ker \psi_1$ contains a nonzero element in $F(H)\times \{0\}$. Since the algebra $F(H)$ is simple, we then have
$\psi_1(F(H)\times \{0\})= 0$. 
In particular,
$\psi_1$ vanishes on $[F(H),F(H)]\times \{0\})$, which
implies 
\[
\psi_2((a,0)) = -\Phi((a,0))= -(a,0),
\] 
for every $a\in [F(H),F(H)]$. However, taking any $a,b\in [F(H),F(H)]$ such that $ab\ne 0$ and $ba=0$,
we arrive at a contradiction since 
\[
(ab,0) = \psi_2((a,0))\cdot \psi_2((b,0))=\psi_2((b,0)(a,0))=0.
\]
Therefore, $\psi_1$ is injective on $F(H)\times \{0\}$. 

Suppose $\psi_1(1)\in M$. Then 
$\psi_1(1)= (e,f)$ where $e$ and $f$ are idempotents in $K(H)$.  As such,
$e$ and $f$ are finite rank operators.
We have
$$\psi_1(x)=\psi_1(1x1)=\psi_1(1)\psi_1(x)\psi_1(1)=(e,f)\psi_1(x)(e,f),$$  for every
$x\in A$, showing that $$\psi_1(A) \subseteq eK(H)e\times fK(H)f.$$ However, $eK(H)e\times fK(H)f$ is  a finite-dimensional space, which contradicts the injectivity of $\psi_1$ on  $F(H)\times \{0\}$.
Therefore, $\psi_1(1)\notin M$.

In a similar fashion we show   that
 $\psi_2$ is injective on $\{0\}\times F(H)$, from which we derive that 
 $\psi_2(1)\notin M$. However, we then have
$\psi_1(1)\cdot \psi_2(1)\ne 0$, which contradicts
$\psi_1(A)\cdot \psi_2(A) = \{0\}$. This proves that  $\Phi$ is not standard. 

Let us add that $\Phi$  does not satisfy the conclusion of Theorem \ref{mt}. 
Indeed, the center of $A$ is equal to $\mathbb C  1$ (so there are no nontrivial central idempotents in $A$) and it is easy to check that $\Phi$ is not of the form 
$\Phi(x)=\lambda \Psi(x) +\zeta(x) 1$ with $\lambda\in \mathbb C$, $\zeta:A\to\mathbb C$ and
$\Psi$ a homomorphism or an antihomomorphism.
Thus, this example also shows that the assumption that algebras have no quotients isomorphic to $\mathbb C$ is necessary in Theorem \ref{mt}.  As mentioned in the introduction, so is the assumption that algebras have no quotients isomorphic to $M_2(\mathbb C)$.
\end{example}


\subsection*{Acknowledgements} The first author was partially supported by the ARIS Grants P1-0288 and J1-60025. The second, third, and fourth author  were partially supported by
MCIN/AEI/10.13039/501100011033 and by ``ERDF A way of making Europe''
grant PID2021-122126NB-C31,
by the IMAG--Mar{\'i}a de Maeztu grant CEX2020-001105-M/AEI/10.13039/
501100 011033, and
by Junta de Andalucía grants no. FQM185 and FQM375.
The second  author was partially supported by Grant FPU21/00617 at University of Granada
founded by Ministerio de Universidades (Spain).

\subsection*{Data availability}

There are no data associates for the submission entitled ``Commutativity preserving mappings in Banach algebras''.

\subsection*{Statements and Declarations}

The authors declare that they have no financial or conflict of interest.

\end{document}